\documentclass[12pt]{article}
\usepackage{hyperref,amssymb,amsmath,graphicx,verbatim,eufrak,amsthm}
\usepackage{fullpage}
\usepackage{color}

\newtheorem{thm}{Theorem}
\newtheorem{thmtool}{Theorem}[section]
\newtheorem{cormain}[thm]{Corollary}
\newtheorem{cor}[thmtool]{Corollary}
\newtheorem{lem}[thmtool]{Lemma}
\newtheorem{prop}[thm]{Proposition}
\newtheorem{clm}[thmtool]{Claim}

\newtheorem*{thm*}{Theorem}

\theoremstyle{definition}
\newtheorem{dfn}[thm]{Definition}
\theoremstyle{remark}

\numberwithin{equation}{section}

\newcommand{\abs}[1]{\left\vert#1\right\vert}

\newcommand{\eps}{\varepsilon}

\renewcommand{\Pr}{}
\let\Pr\relax
\DeclareMathOperator{\Pr}{\mathbb{P}}

\def\squareforqed{\hbox{\rlap{$\sqcap$}$\sqcup$}}
\def\qed{\ifmmode\squareforqed\else{\unskip\nobreak\hfil
\penalty50\hskip1em\null\nobreak\hfil\squareforqed
\parfillskip=0pt\finalhyphendemerits=0\endgraf}\fi}

\begin{document}
\title{Long cycles in subgraphs of (pseudo)random directed graphs}
\author{Ido Ben-Eliezer\thanks{School of Computer Science, Raymond and Beverly Sackler Faculy of Exact Sciences, Tel Aviv University, Tel Aviv 69978, Israel, e-mail: \textbf{idobene@post.tau.ac.il}. Research supported in part by an ERC advanced grant.} \and Michael Krivelevich \thanks{School of Mathematical Sciences, Raymond and Beverly Sackler Faculty of Exact Sciences, Tel Aviv University,
Tel Aviv 69978, Israel, e-mail: \textbf{ krivelev@post.tau.ac.il}.
Research supported in part by USA-Israel BSF grant 2006322 and by grant
1063/08 from the Israel Science Foundation.} \and Benny Sudakov \thanks{Department of Mathematics, UCLA,
Los Angeles, CA 90095. Email: \textbf{bsudakov@math.ucla.edu}.
Research supported in part by NSF CAREER award DMS-0812005 and by a
USA-Israeli BSF grant.}} \maketitle
\begin{abstract}
We study the resilience of random and pseudorandom directed graphs
with respect to the property of having long directed cycles. For
every $0 < \gamma < 1/2$ we find a constant $c=c(\gamma)$ such that
the following holds. Let $G=(V,E)$ be a (pseudo)random directed
graph on $n$ vertices, and let $G'$ be a subgraph of $G$ with
$(1/2+\gamma)|E|$ edges. Then $G'$ contains a directed cycle of
length at least $(c-o(1))n$. Moreover, there is a subgraph $G''$ of
$G$ with $(1/2+\gamma-o(1))|E|$ edges that does not contain a cycle
of length at least $cn$.
\end{abstract}
\section{Introduction}

Given a property $\mathcal P$, a typical problem in extremal graph
theory can be stated as follows. Given a number of vertices $n$,
what is the minimal (or maximal) number $f_{\mathcal P}(n)$ such
that any graph on $n$ vertices with $f(n)$ edges possesses $\mathcal
P$? Many examples of such problems and results can be found, e.g.,
in~\cite{Bollobas04}.

Usually, the property $\mathcal P$ we consider in extremal problems
is either {\em monotone increasing} or {\em monotone decreasing}. A
property $\mathcal P$ is monotone increasing (respectively,
decreasing) if it preserved under edge addition (respectively,
deletion).

The {\em resilience} of a graph $G$ with respect to a property
$\mathcal P$ measures how {\em far} the graph is from any graph $H$
that does not have $\mathcal P$. In particular, the study of
resilience usually focuses on monotone properties, and the following
two types of problems are studied.

\paragraph{Global Resilience.} Given a monotone increasing  property $\mathcal
P$, the global resilience of $G$ respect to $\mathcal P$ is the
maximal integer $R$ such that for every subset $E_0\subseteq E(G)$ of $|E_0|=R$ edges,
the graph $G-E_0$ still possesses $\mathcal P$. For the case of a monotone decreasing property $\mathcal P$,
the global resilience is defined as the maximum number $R$ such that the addition of any subset of $R$
edges to $G$ still results in a graph $G'\in {\mathcal P}$.

One can also define the notion of {\bf local resilience} of a graph with respect to, say, a monotone increasing
property $\mathcal P$ as the maximum number $r$ such that for any subgraph $H\subseteq G$ of maximum degree
$r$, the graph $G-H$ is still in $\mathcal P$. Since in this paper we will be concerned with properties related to global
resilience, we will not dwell on the notion of local resilience anymore.

\medskip

The explicit study of this notion was initiated by Sudakov and
Vu~\cite{SV08}, but in a sense many well known theorems in extremal
graph theory can be stated using this terminology. For example,
given a fixed graph $H$, the  Tur\'{a}n number of $H$, denoted by
$ex(H,n)$, is the minimum number $m$ such that any graph on $n$
vertices with $m$ edges contains a copy of $H$. Clearly, the study
of Tur\'{a}n numbers is equivalent to the study of the global
resilience of the complete graph $K_n$ with respect to the property
of having a copy of $H$.

Woodall~\cite{Woodall72} gave tight bounds for the number of edges
in an undirected graph that guarantees the existence of a cycle of
length at least $\ell$. In our terminology, he gave tight bounds on
the global resilience of $K_n$ with respect to the property of
having a cycle of length at least $\ell$. We will discuss Woodall's
result later and will also use his result in our work.
Lewin~\cite{Lewin75} studied the analogous problem for directed
graphs, and he gave tight bounds on the number of edges required for
having a directed cycle of length at least $\ell$. Many extremal
results regarding the existence of cycles in directed graphs can be
found, e.g., in~\cite{BT81}.

Recently, there has been an extensive line of works that studying
the resilience of graphs with respect to different properties.
Dellamonica et al.~\cite{DKMS08} studied the local and global
resilience of long cycles in pseudorandom undirected graphs.
Krivelevich et al.~\cite{KLS10} studied the resilience with respect
to pancyclicity (having cycle of every possible length). Ben-Shimon
et al.~\cite{BSKS10} studied the resilience of several graph
properties in random regular graphs. Alon and Sudakov~\cite{AS10}
studied the resilience of the chromatic number in random graphs.
Huang et al.~\cite{HLS10} studied the resilience with respect to
having a spanning graph $H$ as a subgraph, answering a question of
B\"{o}ttcher et al.~\cite{BKT10}. Balogh et al. \cite{BCS} studied
the resilience of random and pseudorandom graphs with respect to
containing a copy of a given nearly spanning tree of bounded maximum
degree.

Here we study the resilience of pseudorandom (and hence, of random)
directed graphs with respect to the property of having a long
directed cycle (namely, a simple directed cycle that covers a
constant fraction of the vertices). We prove asymptotically tight
bounds, and thus provide the asymptotic value of the resilience of
every graph with respect to this property, assuming it has some
predefined pseudorandomness property. Our proof applies a variant of
the celebrated Szemer\'{e}di's regularity lemma for sparse directed
graphs, and a short and simple technique for finding a long directed
path in pseudorandom directed graphs. Using these techniques we can
reduce our problem to the case of undirected graphs, where by
applying techniques of~\cite{DKMS08} we can give tight bounds.

\subsection{The models}
We consider here directed graphs on $n$ vertices, where antiparallel
edges are allowed. We say that a graph $D=(V,E)$ has density $p$ if
$|E| = pn^2$.

Let $D(n,p)$ be the following natural distribution of random
directed graphs. Every graph in the support of $D(n,p)$ contains $n$
vertices, and for every two distinct vertices $x,y$, there is an
edge from $x$ to $y$ with probability $p$, and independently there
is an edge from $y$ to $x$ with probability $p$. Clearly, the
expected number of edges is $2p{n \choose 2}$.

Once we define our random digraph model, it is usually desirable to
define a pseudorandom analog. That is, we would like to define a
property such that graphs with this property have many of the 'nice'
properties of random graphs. Roughly speaking, we say that a
directed graph is pseudorandom if the number of edges between every
two large enough sets is close to the expected number of edges in a
random directed graph with the same density. More formally, we say
that a directed graph $G$ is $(p,r)$-\textit{pseudorandom} if it has
edge density $p$ and for every two disjoint sets $A,B \subseteq
V(G)$, $|A|=|B|$, the number of edges from $A$ to $B$, denoted by
$e_G(A,B)$, satisfies
$$ \abs{e_G(A,B)-p|A||B|} \leq r|A|\sqrt{pn}.$$

This is (up to normalization) a directed variant of the well known notion of {\em jumbled
graphs}, that was introduced By Thomason~\cite{Thomason85}. In his
celebrated work, Thomason essentially proved that a graph
distributed as $G(n,p)$  is
$(p,O(1))$-pseudorandom with high probability.\footnote{Here a
sequence of events $A_n$, $n \geq 1$ is said to occur with high
probability  if $\lim_{n \rightarrow \infty} \Pr[A_n] = 1.$ } On
the other hand, there is no infinite sequence of
$(p,o(1))$-pseudorandom graphs.

The following lemma can be easily verified by combining a Chernoff
type bound with the union bound.
\begin{lem}
\label{lem:random-graph-is-uniform} For every constant $c>0$ there
is a constant $C>0$ such that for $p \geq \frac{C}{n}$, a random
directed graph $G \in D(n,p)$ is $(p,c)$-pseudorandom with high
probability.
\end{lem}

Our results in this work will hold for every $(p,r)$-pseudorandom
graph with $p \geq \frac{C}{n}$ for some sufficiently large constant
$C$, and every $r \leq \mu \sqrt{pn}$ for some small constant $\mu
>0$ that will be chosen later. By Lemma~\ref{lem:random-graph-is-uniform}, a random directed
graph distributed according to $D(n,p)$ with $p \geq \frac{C}{n}$
has this property with high probability.

We show here that the directed case is both similar and different
from the undirected case. In fact, since we reduce here the directed
case to the global resilience problem of the undirected case, we can
use ideas from Dellamonica et al.~\cite{DKMS08} in order to get our
bounds on the resilience for directed graphs. On the other hand,
many of the techniques that were used for the undirected case cannot
be applied in the directed case. Also, the range of parameters
relevant to us is rather different, since in particular the result
of Dellamonica et al.~\cite{DKMS08} shows that the removal of any
$0.99$-fraction of the edges of a (pseudo)random undirected graph
still leaves a cycle of linear size. For the directed case it is
easy to see that one can always remove half of the edges of any
directed graph and get an acyclic directed graph, and hence a graph
with no cycles at all.

\subsection{Our results}
\label{subsection:our-results}

Woodall~\cite{Woodall72} studied the minimal number of edges that
guarantees the existence of a long cycle. In our terminology, he studied the
global resilience of the complete graph $K_n$ with respect to the
property of having a cycle of length at least $\ell$. He proved the
following.

\begin{thm}[Woodall~\cite{Woodall72}] Let $3 \leq \ell \leq n$. Every
graph $G$ on $n$ vertices satisfying $$e(G) \geq \left \lceil
\frac{n-1}{\ell-2} \right \rceil \cdot {\ell-1 \choose 2}+{r+1
\choose 2}+1,$$ where $r = (n-1) \mod{(\ell-2)}$, has a cycle of
length at least $\ell$.
\end{thm}

It is easy to verify that Woodall's bound is best possible. Indeed,
take a graph formed by $\lceil \frac{n-1}{\ell -2} \rceil$ disjoint
cliques of size $\ell-2$, a single smaller clique of size $r$ and a
vertex that is connected to every other vertex in the graph.
Clearly, the length of the longest cycle in this graph is at most
$\ell-1$.

The work of Dellamonica et al.~\cite{DKMS08} can be viewed as a
generalization of Woodall's work from the case of $K_n$ to the case
of general pseudorandom graphs. In order to cite their result and
also for future reference in our paper the following function is
defined.

\begin{dfn}
For a given $0 \leq \alpha < 1$, define $$w(\alpha) =
1-(1-\alpha)\lfloor (1-\alpha)^{-1} \rfloor.$$ \label{dfn:w(alpha)}
\end{dfn}
It is easy to verify that we have $w(0) = 0$ and $\lim_{x \nearrow
1} w(x) = 0$. The following asymptotic version of Woodall's result
is proved in~\cite{DKMS08} and is easier to work with.

\begin{thm}[~\cite{DKMS08}]\label{th3}
\label{thm:Woodall-new} Let $\alpha>0$. For every $\beta>0$ there is
$n_0$ such that for every graph $G$ on $n>n_0$ vertices satisfying
 $$|E(G)| \geq {n \choose 2} \cdot \Big(1-(1-w(\alpha))(\alpha+w(\alpha))+\beta
 \Big)$$
has a cycle of length at least $(1-\alpha) \cdot n$.
\end{thm}

Dellamonica et al. proved in \cite{DKMS08} that Theorem \ref{th3}
can be extended to (sparse) pseudorandom graphs; more specifically
they proved that any subgraph $G'=(V,E')$ of a pseudorandom graph
$G=(V,E)$ with $|E'|\ge (1-(1-w(\alpha))(\alpha+w(\alpha))+o(1))|E|$
edges has a cycle of length at least $(1-\alpha) \cdot |V|$. Here we
obtain a directed analogue of their result.

We can now state our main theorem.

\begin{thm}\label{th4}
Fix $0<\gamma<\tfrac{1}{2}$ and let $G=(V,E)$ be a
$(p,r)$-pseudorandom directed graph on $n$ vertices, where $r \leq
\mu \sqrt{np}$ and $\mu(\gamma)>0$ is a sufficiently small constant
that depends only on $\gamma$ and $n$ is sufficiently large. Let
$G'$ be the subgraph of $G$ with at least $(\frac{1}{2}+\gamma)|E|$
edges. Then $G'$ contains a directed cycle of length at least
$(1-\alpha-o(1))\cdot n$, where $\alpha$ satisfies $$2\gamma = 1 -
(1-w(\alpha))(\alpha+w(\alpha)) .$$ \label{thm:main-theorem}
\end{thm}

Observe crucially that every directed graph $G=(V,E)$ contains an
acyclic subgraph $G'$ with at least $|E|/2$ edges. Indeed, fix a
permutation $\sigma: V \to V$, and let $G_1$ be the subgraph with
all edges $xy$ such that $\sigma(x)>\sigma(y)$, and $G_2$ be the
subgraph with all edges $xy$ such that $\sigma(x)<\sigma(y)$. Then
both $G_1$ and $G_2$ are acyclic, and at least one of them contains
at least half of the edges of $G$.

Theorem \ref{th4} yields the following two immediate corollaries.

\begin{cormain}
For every $\gamma>0$ there is a constant $c_1(\gamma)>0$ such that
the following holds. Let $G$ be a $(p,r)$-pseudorandom graph on $n$
vertices, $r \leq \mu \sqrt{np}$ where $\mu(\gamma)>0$ is some
sufficiently small constant that depends only on $\gamma$ and $n$ is
sufficiently large. Let $G'$ be a subgraph of $G$ with at least
$(1/2+\gamma)|E(G)|$ edges. Then $G'$ contains a directed cycle of
length at least $c_1 n$.
\end{cormain}

In words, the above corollary guarantees that the deletion of less
than half of the edges of a pseudorandom digraph leaves a cycle of
linear length.

\begin{cormain}
There exists a function $c_2(\epsilon)$ with $\lim_{\epsilon\to
0}c_2(\epsilon)=0$ such that the following holds. Let $G$ be a
$(p,r)$-pseudorandom graph on $n$ vertices, $r \leq \mu \sqrt{np}$
where $\mu(\gamma)>0$ is some sufficiently small constant that
depends only on $\gamma$ and $n$ is sufficiently large. Let $G'$ be
a subgraph of $G$ with at least $(1-\eps)|E(G)|$ edges. Then $G'$
contains a directed cycle of length at least $(1-c_2) \cdot n$.
\end{cormain}

Here, we have that deleting a negligible fraction of the edges of a
pseudorandom digraph leaves a cycle of length close to $n$.

Finally, we prove the following matching lower bound.

\begin{prop}
Fix $0<\gamma<\tfrac{1}{2}$ and let $G$ be a $(p,r)$-pseudorandom
directed graph on $n$ vertices, where $r = O(\sqrt{np})$ and $pn \to
\infty$. There is a subgraph $G'$ with $(\frac{1}{2}+\gamma)|E|$
edges that does not contain any directed cycle of length
$(1-\alpha+o(1))\cdot n$, where $\alpha$ satisfies
$$2\gamma = 1 - (1-w(\alpha))(\alpha+w(\alpha)) .$$
\label{prop:lower-bound}
\end{prop}

\paragraph{Our Tools.} One of the main tools we use in this work is a
sparse directed variant of Szemer\'{e}di's regularity lemma
(Lemma~\ref{lem:regularity}), that was stated in~\cite{DK02}. This
allows us to partition our graph to a constant number of regular
pairs, and essentially to reduce the problem to finding almost
spanning paths in regular pairs.

To this end, we use a simple yet powerful lemma that finds almost
spanning paths in expanding graphs
(Lemma~\ref{lem:long-path-in-expander}). In our case, a regular pair
is a bipartite expander is both directions. The approach is based on
ideas from~\cite{BBDK06,BB07,BKS10}.

\medskip

The rest of the paper is organized as follows. In
Section~\ref{sec:tools} we state the sparse directed regularity
lemma, and prove that regular pairs have an almost spanning directed
path. In Section~\ref{sec:proof-main}, we reduce the resilience
problem in directed graphs to undirected graphs, and then apply
ideas from~\cite{DKMS08}. In Section~\ref{sec:lower-bounds} we show
that our results are essentially tight.

Throughout the proofs we assume that the order of $G$, denoted by
$n$, is large enough. We do not try to optimize constants and omit
floor and ceil signs whenever these are not crucial.

\section{The regularity lemma for sparse directed graphs and long paths in regular pairs}
\label{sec:tools}

\subsection{The regularity lemma}

In this section we follow~\cite{DK02} and state a regularity lemma
for sparse directed graphs. We first provide some notation.

Given a directed graph $G=(V,E)$, for any pair of disjoint sets of
vertices $U,W$, we let $E_G(U,W)$ be the set of edges directed from
$U$ to $W$, and let $e_G(U,W) = |E_G(U,W)|$. We say that $G$ is
$(\delta,D,p)$-\textit{bounded} if for any two sets $U,W$ such that
$|U|,|W| \geq \delta |V|$ we have $$e_G(U,W) \leq Dp|U||W|.$$

The {\em edge density} from a set $U$ to a set $W$ is defined by
$\frac{e_G(U,W)}{|U||W|}$. We say that two sets $U$ and $W$ span  a bipartite directed graph
of {\em bi-density} $p$ if it has edge density at least $p$ in both
directions. Also define the {\em directed $p$-density} from $U$ to
$W$ by
$$d_{G,p}(U,W) = \frac{e_G(U,W)}{p|U||W|}.$$
We omit the index graph $G$ and write $d_p(U,W)$ whenever the base graph is clear from
the context.

For $0 < \delta \leq 1$, a pair $(U,W)$ is
$(\delta,p)$-\textit{regular} in a digraph $G$ if for every $U'
\subseteq U$ and $W' \subseteq W$ such that $|U'| \geq \delta  |U|$
and $|W'| \geq \delta |W|$ we have both
$$|d_{G,p}(U,W)-d_{G,p}(U',W')| < \delta,$$
and
$$|d_{G,p}(W,U)-d_{G,p}(W',U')| < \delta.$$

A partition $\mathcal P = \{V_0,V_1,\ldots,V_k\}$ of $V$ is
$(\delta,k,p)$-\textit{regular} if the following properties hold.
\begin{enumerate}
\item $|V_0| \leq \delta |V|$.
\item $|V_i|=|V_j|$ for all $1 \leq i < j \leq k$.
\item At least $(1-\delta){k \choose 2}$ of the pairs $(V_i,V_j), 1 \leq i < j \leq k$,
are $(\delta,p)$-regular.
\end{enumerate}

We will use the following variant of Szemer\'{e}di's regularity
lemma, that is stated in~\cite{DK02}, and its proof follows lines
similar to the proof of the regularity lemma for sparse graphs,
proved independently by Kohayakawa and by R\"{o}dl (see,
e.g.,~\cite{Kohayakawa97}). In~\cite{DK02} the lemma is stated for
oriented graphs (where no anti parallel edges are allowed), yet the
result can be easily adjusted to our case, where anti parallel edges
are allowed.

\begin{lem}[Lemma 3 in ~\cite{DK02}]
For any real number $\delta>0$, an integer $k_0 \geq 1$ and real
number $D>1$, there exist constants $\eta=\eta(\delta,k_0,D)$ and
$K=K(\delta,k_0,D)\geq k_0$ such that for any $0 < p(n) \leq 1$, any
$(\eta,D,p)$-bounded directed graph $G$ admits a
$(\delta,k,p)$-regular partition for some $k_0 \leq m \leq K$.
\label{lem:regularity}
\end{lem}

\subsection{Every regular pair contains a long path}

We next prove that every regular pair of positive bi-density contains
an almost spanning path. To this end, we first show a trivial
expansion property of regular pairs, and then apply this property to
prove the desired result.

\begin{clm}
Let $(U,W)$ be a $(\delta,p)$-regular pair for $|U|=|W|$ with
bi-density at least $2\delta p$, where $p>0$. Then for every two
sets $U' \subseteq U$ and $W' \subseteq W$ such that $|U'| \geq
\delta |U|$ and $|W'| \geq \delta |W|$ there is a directed edge from
$U'$ to $W'$. \label{clm:edge-between-small-sets}
\end{clm}
\begin{proof}
By regularity we have $$e_G(U',W') \geq (d_p(U,W)-\delta)p|U'||W'|
\geq (2\delta-\delta)p|U'||W'| =\delta p  |U'||W'| >0.$$ The claim
follows.
\end{proof}

We next show that a bipartite directed graph with a simple expansion
property contains a long directed path. The proof follows ideas
from~\cite{BBDK06,BB07,BKS10}.

\begin{lem}
Let $H = (V_1,V_2,E)$, $|V_1| = |V_2| = t$, be a directed bipartite
graph that satisfies the following property: for every two sets $A
\subseteq V_1, B \subseteq V_2$ of size $k$, there is at least one
edge from $A$ to $B$ and there is at least one edge from $B$ to $A$.
Then $H$ contains a directed path of length $2t-4k+3$.
\label{lem:long-path-in-expander}
\end{lem}
\begin{proof}
Recall that the DFS (Depth First Search) is a graph search algorithm
that visits all the vertices of a (directed or undirected) graph $G$
as follows. It maintains three sets of vertices,  letting $S$ be the
set of vertices which we have completed exploring, $T$ be the
set of unvisited vertices, and $U = V(G) \setminus (S \cup T)$,
where the vertices of $U$ are kept in a {\em stack} (a {\em last in,
first out} data structure). It is also assumed that some order
$\sigma$ on the vertices of $G$ is fixed, and the algorithm
prioritizes vertices according to $\sigma$. The DFS starts with
$S=\emptyset,U=\{\sigma_1\}$ and $T = V(G) \setminus \{\sigma_1\}$.

While there is a vertex in $V(G) \setminus S$, if $U$ is non-empty,
let $v$ be the last vertex that was added to $U$. If $v$ has an
out-neighbor $u \in T$, the algorithm inserts $u$ to $U$. If $v$ does not have an out-neighbor in $T$ then $v$ is
popped out from $U$ and is moved to $S$. If $U$ is empty, the
algorithm chooses an arbitrary vertex from $T$ and pushes it to $U$.

We now proceed to the proof of the lemma. We execute the DFS
algorithm for an arbtrary chosen order $\sigma$ on the vertices of the graph.
We let again $S,T,U$ be three sets of vertices as defined
above. At the beginning of the algorithm, all the vertices are in
$T$, and at each step a single vertex either moves from $T$ to $U$
or from $U$ to $S$. At the end of the algorithm, all the vertices
are in $S$.

Consider the point during the execution of the algorithm when  $|S|
= |T|$. Observe crucially that all the vertices in $U$ form a
directed path, and we have $ \left| |U \cap V_1| - |U\cap V_2|
\right| \leq 1$. Since $|U|=2t-|S|-|T|=2t-2|S|$ is even, we have in fact
$$
|U\cap V_1|=|U\cap V_2|\,.
$$


We get that $$ |S| = |S \cap V_1| + |S \cap V_2| = |T \cap V_1| + |T
\cap V_2| = |T|,$$

and $$|V_1 \setminus U|=|S \cap V_1| + |T \cap V_1| = |S \cap V_2| + |T \cap V_2| =
|V_2 \setminus U|.$$ Hence, we get both $$ |S
\cap V_2| = |T \cap V_1| \,,$$ and $$ |S \cap V_1| = |T \cap
V_2|\,.$$

Assume without loss of generality that $$|S \cap V_1| \geq |S|/2
\geq |S \cap V_2|.$$ Then $|S \cap V_1| \geq t/2-|U|/4$ and
therefore $|T \cap V_2| \geq t/2-|U|/4$. Observe crucially that
there are no edges from $S$ to $T$. By the assumption of the lemma
we conclude $|S \cap V_1|, |T \cap V_2| \le k-1$ and therefore we
get $t/2-|U|/4 \le k-1$ and hence $|U| \ge 2t-4k+4$. Thus $H$
contains a directed path $|U|$ of length $2t-4k+3$, as desired.
\end{proof}

We therefore have the following corollary.
\begin{cor}
Let $(U,W)$ be a $(\delta,p)$-regular pair with bi-density at least
$2\delta p$ and $|U| = |W| = t$, $p>0$. Then the bipartite directed graph
between $U$ and $W$ contains a directed path of length $(1-2\delta)
\cdot 2t+2$ that starts at $U$.
\label{cor:regular-pair-has-long-path}
\end{cor}
\begin{proof}
By Claim~\ref{clm:edge-between-small-sets}, there is an edge in each
direction between every two sets of size $\delta t$ in $U$ and $W$.
Therefore Lemma~\ref{lem:long-path-in-expander} implies the
existence of a directed path of length $(1-2\delta)2t+3$. Note that
if the first vertex in the path is from $W$ we may remove it, thus
getting a directed path of length at least $(1-2\delta)2t+2$ that
starts at $U$.
\end{proof}

\section{Proof of Theorem~\ref{thm:main-theorem}}
\label{sec:proof-main}

In this section we prove our main result.
Given a constant $\gamma>0$, we essentially want to prove that every
subgraph with $(1/2+\gamma)$-fraction of the edges of a pseudorandom
directed graph contains a long directed cycle. Let $\delta =
\delta(\gamma)$ to be fixed later, $K=K(\delta,1/\delta,1+\delta)$
and $\eta=\eta(\delta,1/\delta,1+\delta)$ be the constants defined
by the regularity lemma (Lemma~\ref{lem:regularity}).

Let $G = (V,E)$ be a $(p,r)$-pseudorandom directed graph with
$r \leq \mu\sqrt{np}$. For $\mu
\leq \delta \cdot \min\{\eta, 1/K\}$ we have
$$r \leq \delta \sqrt{np} \cdot \min\{\eta, 1/K\}.$$ Let $A,B$ be
two sets of vertices of size $\eta n$ in $G$. Observe that
$$r \sqrt{pn|A||B|} \leq \delta \sqrt{np}\, \eta \,\sqrt{p \eta^2 n^3}=\delta\eta^2n^2p\leq
\delta \eta ^2 n^2 = \delta |A||B|.$$
Therefore, we get that $G$ is
$(\eta,1+\delta,p)$-bounded.

Given a subgraph $G' = (V,E')$ of $G$ that contains
$(1/2+\gamma)|E|$ edges, our goal is to show that $G'$ contains a
long directed cycle.

Clearly, $G'$ is $(\eta,1+\delta,p)$-bounded as well, and hence we
can apply the sparse directed regularity lemma
(Lemma~\ref{lem:regularity}) to $G'$ and get a partition of $V$ to
clusters $V_0,V_1,\ldots,V_m$, where $1 / \delta \le m \leq K$,
$|V_0| \leq \delta n$, $|V_1|=|V_2|=\cdots=|V_m|=t$ and all but at
most $\delta$-fraction of the pairs $(V_i,V_j)$ are
$(\delta,p)$-regular. Note that $ \frac{n(1-\delta)}{m} \leq t \leq \frac{n}{m}$.

We next define an \textbf{undirected} auxiliary graph $H$ on the
clusters $V_1,\ldots,V_m$. With a slight abuse of notation, we
denote the vertices of $H$ by $V_1,V_2,\ldots,V_m$. Two vertices
$V_i$ and $V_j$ are connected if the pair $(V_i,V_j)$ is
$(\delta,p)$-regular and has bi-density at least $2\delta p$.

Since $G$ is $(p,r)$-pseudorandom and $r \leq \frac{\delta \cdot
\sqrt{np}}{K} \leq \frac{\delta \cdot \sqrt{np}}{m}$, we get that the edge density of every pair
$(V_i,V_j)$ in $G$ is at least $1-\delta$ and at most $1+\delta$.

Observe that if $V_i$ and $V_j$ are not connected by an edge in $H$,
one of the following must happen.

\begin{itemize}
\item The pair $(V_i,V_j)$ is not regular.
\item Either $|E_{G'}(V_i,V_j)| < 2\delta p|V_i||V_j|$ or $|E_{G'}(V_j,V_i)| < 2\delta p|V_i||V_j|$ . In other words, at least $(1-3\delta)p|V_i||V_j|$ of the edges from $V_i$ to $V_j$ or from $V_j$ to $V_i$ in $G$ are not in $G'$.
\end{itemize}

The number of non-regular pairs is at most $\delta {m \choose 2}$.
Since in every pair with bi-density less than $2 \delta p$
at least $(1-3\delta)p t^2$ edges were lost when moving from $G$ to $G'$, we get that the number of pairs
in $G'$ with edge density less than $2\delta p$ is bounded by
$$\frac{(1/2-\gamma) p n^2}{(1-3\delta)p t^2} \leq \frac{(1/2-\gamma)pn^2}{(1-3\delta)p(\frac{n(1-\delta)}{m})^2} = \frac{1/2-\gamma}{(1-3\delta)(1-\delta)^2} \cdot m^2.$$

Observe that by our choice of parameters for the regularity lemma we
have $m \geq \frac{1}{\delta}$, and hence we have
$$ m^2 \leq 2{m \choose 2}(1+2/m) \leq 2{m \choose 2}(1+2\delta).$$

We conclude that the fraction of non-edges in $H$ is bounded by
$$ \delta + \frac{(1-2\gamma)(1+2\delta)}{(1-3\delta)(1-\delta)^2}.$$

By taking $\delta \ll \gamma$, we get that the fraction of edges in
$H$ is at least $2\gamma (1-o(1))$, where the $o(1)$ term depends
only on $\delta$ and can be made arbitrary small.

Let $\alpha$ be the minimal solution of the equation
$$ 2\gamma = 1 - (1-w(\alpha))(\alpha+w(\alpha)),$$ where $w$ is the function defined in
Subsection~\ref{subsection:our-results}. By
Theorem~\ref{thm:Woodall-new}, we get that $H$ contains an
undirected cycle of length $(1-\alpha-o(1))m$.

We complete the proof by showing how to construct a long cycle in
$G'$ given a long cycle in $H$. We start with the case that $H$ contains
a long cycle of even length, and later show how to modify
the proof in the case of odd length.

Let $V_{i_1},V_{i_2},\ldots V_{i_{2b}}$ be a cycle of length
$(1-\alpha+o(1))m$ in $H$. Note that for every $1 \leq q \leq 2b$,
the pair $(V_{i_q},V_{i_{q+1}})$ (where we identify $V_{i_{2b+1}}$
with $V_{i_1}$) is $(\delta,p)$-regular and has edge bi-
density at
least $2\delta p$. Therefore, by
Corollary~\ref{cor:regular-pair-has-long-path}, for every $1 \leq q
\leq b$, there is a directed path $P_q$ of length at least
$(1-2\delta) \cdot 2t$ that alternates between $V_{i_{2q-1}}$ and
$V_{i_{2q}}$. For every $1 \leq q \leq b$, let $P_q^{R} = P_q \cap
V_{i_{2q}}$ and let $P_q^{L} = P_q \cap V_{i_{2q-1}}$.

Observe that for every $1 \leq q \leq b$, by
Claim~\ref{clm:edge-between-small-sets} there is an edge from the
last $\delta t$ vertices of $P_q^{R}$ to the first $\delta t$
vertices if $P^L_{q+1}$ (here we identify $P^L_{b+1}$ with
$P^L_{1}$). Thus we can paste every two consecutive paths together,
losing at most $2\delta t$ vertices from each path. We conclude that
we can paste all the paths together to get a directed cycle of
length
$$b \cdot (1-2\delta-\delta) \cdot 2t = (1-\alpha-o(1))mt= (1-\alpha-o(1))
|V(G')|,$$ as claimed.

Finally, suppose there is an odd cycle of length $(1-\alpha+o(1))m$
in $H$, whose vertices are by $V_{i_1},V_{i_2},\ldots V_{i_{2b+1}}$.
For every $1 \leq q \leq b$, let $P_q$ be a path of length at
least$(1-2\delta) \cdot 2t$ that alternates between $V_{i_{2q-1}}$
and $V_{i_{2q}}$. Moreover, let $$V'_{i_{2b+1}} = \{ v \in
V_{i_{2b+1}} : \forall u \in P^{R}_b, (u,v) \notin E \},$$

and

$$V''_{i_{2b+1}} = \{ v \in V_{i_{2b+1}} : \forall u
\in P^{L}_1, (v,u) \notin E \}.$$

By Claim~\ref{clm:edge-between-small-sets}, we have
$|V'_{i_{2b+1}}|, |V''_{i_{2b+1}}| \leq \delta t$, and therefore for
all but $2 \delta t$ of the vertices in $V_{i_{2b+1}}$ have an edge
from the last $2 \delta t$ vertices in $P_b$ and an edge to the
first $2 \delta t$ first vertices in $P_1$. Since $t > 2\delta t$ we
conclude that we can connect $P_b$ and $P_1$ through a vertex in
$V_{i_{2b+1}}$, thus getting a path of length
$$ b \cdot (1-2\delta-\delta)2t = (1-\alpha+o(1))
|V(G')|.$$
Theorem~\ref{thm:main-theorem} follows. \qed

\section{Lower bounds}
\label{sec:lower-bounds}

Let $G=(V,E)$ be a directed graph. Recall that by fixing a
permutation $\sigma$ on the vertices, we can partition the edges of
$G$ to two acyclic sets as follows. The first set contains all
directed edges $xy$ where $\sigma(x)>\sigma(y)$, and the second set
contains all directed edges $xy$ where $\sigma(y)>\sigma(x)$.
Therefore, the global resilience of every directed graph with
respect to the property of having directed cycles is at most $1/2$.
Here we extend this idea and show that our main result is
asymptotically tight.

\paragraph{Proof of Proposition~\ref{prop:lower-bound}.}
We show that there is a subgraph $G'$ with
$(1/2+\gamma)$-fraction of the edges, whose longest directed cycle is
of length at most $(1-\alpha+o(1))n$. Our approach
follows~\cite{DKMS08}.

Recall that $G$ is $(p,r)$-pseudorandom with $r \leq \mu \sqrt{np}$
and $pn \to \infty$. We first claim that for every two disjoint sets
$A,B$ of size $\Omega(n)$, the number of edges from $A$ to $B$ is
$p|A||B|(1+o(1))$. Indeed, let $|A| = an$ and $|B| = bn$, and
suppose that $a<b$. Let $B'$ be a random subset of $B$ of size $an$,
then by linearity of expectation the number of edges between $A$ and
$B'$ is $E(A,B) \cdot \frac{|B'|}{B}$. Therefore, if the number of
edges between $A$ and $B$ is smaller than (respectively, larger
than) from $p|A||B|(1+o(1))$ then there is a choice of a set $B'$
such that $A$ and $B'$ contradicts the $(p,r)$-pseudorandomness of
$G$.

We partition the vertices of $G$ into $k$ classes
$V_1,V_2,\ldots,V_k$ of size $(1-\alpha)n$ each, and one additional
class $V_{k+1}$ of size $w(\alpha) n \leq (1-\alpha)n$. Let $G'$ be
the subgraph with all edges from $V_i$ to $V_j$, for $1 \leq i < j
\leq k+1$, and all the edges within each class $V_i$. Clearly, $G'$
does not contain a cycle longer than $(1-\alpha)n$, since a directed
path leaving a certain $V_i$ cannot return there. Therefore, we can
conclude the proof by showing that the number of edges in $G'$ is as
claimed.

Indeed, since $G$ is $(p,r)$-pseudorandom and each $V_i$ is of size
$\Omega(n)$, we get that for every $1 \leq i <j \leq k$, the number
of edges from $V_j$ to $V_i$ is $(1-\alpha)^2 \cdot pn^2 (1+o(1))$.
The number of edges from $V_{k+1}$ to $\cup_{i \leq k} V_{i}$ in $G$
is $k(1-\alpha)w(\alpha) \cdot pn^2 (1+o(1))$. Recalling that by
definition of $w(x)$  we have $(1-\alpha)k = 1-w(\alpha)$, we get
that the number of edges we deleted from $G$ to get $G'$ is

\begin{align*}
& \left({k \choose 2}(1-\alpha)^2+k(1-\alpha)w(\alpha)\right) \cdot pn^2 (1+o(1)) \\ & = (1-\alpha)k((1-\alpha)(k-1)+2w(\alpha))(1+o(1))\frac{pn^2}{2}  \\
& = (1-w(\alpha))(\alpha-w(\alpha)+2w(\alpha))\cdot
\frac{pn^2(1+o(1))}{2}
\\ & = (1-w(\alpha))(\alpha+w(\alpha))+o(1))\frac{pn^2}{2}.
\end{align*}

Note that the number of edges in $G$ is $pn^2(1+o(1))$.
Let the number of edges in $G'$ be $(1/2 +\gamma) \cdot pn^2$.
Then the number of edges we deleted satsifies
$$(1+o(1))(1/2-\gamma)pn^2 =
(1-w(\alpha))(\alpha+w(\alpha))+o(1))\frac{pn^2}{2},$$ and therefore $\gamma$ satisfies
$$2\gamma =  1 - (1-w(\alpha))(\alpha+w(\alpha))+o(1)),$$
as claimed. The statement follows.  \qed

\section{Concluding remarks}

We studied the global resilience of pseudorandom directed graphs
with respect to the property of having a long directed cycle. We
gave matching lower and upper bounds, and our proof essentially
reduced our problem to case of undirected graphs.

A variety of questions regarding the resilience of directed graphs
can be asked. A few, somewhat arbitrary examples are the problem of
local resilience with respect to having a long directed cycle, the
resilience with respect to the property of having some fixed
directed graph. Another interesting problem is the resilience with
respect to Hamilitonicity, which in the dense case is settled
in~\cite{HSS10}.

In this work we considered subgraphs with $(1/2+\gamma)$-fraction of
the edges, and observed that every directed graph contains an
acyclic subgraph with $1/2$-fraction of the edges. In~\cite{BKS10},
the authors proved that every two-coloring of the edges of a
pseudorandom digraph contains a relatively long monochromatic path.
That is, instead of proving that a large subgraph has a certain
property, it is proved that every partition of the edges of the
graph has a certain property. It will be interesting to give such
results for other properties of directed graphs.

\end{document}